\documentclass{article}
\usepackage{amsmath,amssymb,amsthm}
\usepackage{color}
\usepackage{url}
\usepackage{enumerate,enumitem}

\usepackage[title]{appendix}  

\usepackage{graphics}
\usepackage{caption}
\usepackage{xcolor}
\usepackage{pgf}

\usepackage[plainpages=false,pdfpagelabels]{hyperref}

\allowdisplaybreaks

\newcommand{\half}{\frac{1}{2}}
\newcommand{\ii}{\mathcal{I}}
\newcommand{\bby}{\mbox{{\boldmath $\mathcal{Y}$}}}
\newcommand{\bbw}{\mbox{{\boldmath $\mathcal{W}$}}}
\newcommand{\by}{\mathbf{Y}}
\newcommand{\bw}{\mathbf{W}}
\newcommand{\rr}{\mathbb{R}}

\newcommand{\eps}{\varepsilon}

\renewenvironment{proof}[1][\proofname]{\par \normalfont \trivlist
\item[\hskip\labelsep\itshape #1]\ignorespaces
}{%
\hspace*{\fill}$\Box$ \endtrivlist }
\renewcommand{\proofname}{{\bf Proof}}

\makeatletter
\def\newrmtheorem#1{\@ifnextchar[{\@rmothm{#1}}{\@rmnthm{#1}}}

\def\@rmnthm#1#2{%
\@ifnextchar[{\@rmxnthm{#1}{#2}}{\@rmynthm{#1}{#2}}}

\def\@rmxnthm#1#2[#3]{\expandafter\@ifdefinable\csname #1\endcsname
{\@definecounter{#1}\@addtoreset{#1}{#3}%
\expandafter\xdef\csname the#1\endcsname{\expandafter\noexpand
  \csname the#3\endcsname \@rmthmcountersep \@rmthmcounter{#1}}%
\global\@namedef{#1}{\@rmthm{#1}{#2}}\global\@namedef{end#1}{\@endrmtheorem}}}

\def\@rmynthm#1#2{\expandafter\@ifdefinable\csname #1\endcsname
{\@definecounter{#1}%
\expandafter\xdef\csname the#1\endcsname{\@rmthmcounter{#1}}%
\global\@namedef{#1}{\@rmthm{#1}{#2}}\global\@namedef{end#1}{\@endrmtheorem}}}

\def\@rmothm#1[#2]#3{\expandafter\@ifdefinable\csname #1\endcsname
  {\global\@namedef{the#1}{\@nameuse{the#2}}%
\global\@namedef{#1}{\@rmthm{#2}{#3}}%
\global\@namedef{end#1}{\@endrmtheorem}}}

\def\@rmthm#1#2{\refstepcounter
    {#1}\@ifnextchar[{\@rmythm{#1}{#2}}{\@rmxthm{#1}{#2}}}

\def\@rmxthm#1#2{\@beginrmtheorem{#2}{\csname the#1\endcsname}\ignorespaces}
\def\@rmythm#1#2[#3]{\@opargbeginrmtheorem{#2}{\csname
       the#1\endcsname}{#3}\ignorespaces}

\def\@rmthmcounter#1{.\noexpand\arabic{#1}}
\def\@rmthmcountersep{}
\def\@beginrmtheorem#1#2{\rm \trivlist
      \item[\hskip \labelsep{\bf #1\ #2\thmrmcounterend}]}
\def\@opargbeginrmtheorem#1#2#3{\rm \trivlist
      \item[\hskip \labelsep{\bf #1\ #2\ (#3)\thmrmcounterend}]}
\def\@endrmtheorem{\endtrivlist}
\def\thmrmcounterend{\hskip 0em\relax}
\def\newrmwntheorem#1#2{\expandafter\@ifdefinable\csname #1\endcsname%
\global\@namedef{#1}{\@rmwnthm{#1}{#2}}%
\global\@namedef{end#1}{\@endrmwntheorem}}

\def\newsltheorem#1{\@ifnextchar[{\@slothm{#1}}{\@slnthm{#1}}}

\def\@slnthm#1#2{%
\@ifnextchar[{\@slxnthm{#1}{#2}}{\@slynthm{#1}{#2}}}

\def\@slxnthm#1#2[#3]{\expandafter\@ifdefinable\csname #1\endcsname
{\@definecounter{#1}\@addtoreset{#1}{#3}%
\expandafter\xdef\csname the#1\endcsname{\expandafter\noexpand
  \csname the#3\endcsname \@slthmcountersep \@slthmcounter{#1}}%
\global\@namedef{#1}{\@slthm{#1}{#2}}\global\@namedef{end#1}{\@endsltheorem}}}

\def\@slynthm#1#2{\expandafter\@ifdefinable\csname #1\endcsname
{\@definecounter{#1}%
\expandafter\xdef\csname the#1\endcsname{\@slthmcounter{#1}}%
\global\@namedef{#1}{\@slthm{#1}{#2}}\global\@namedef{end#1}{\@endsltheorem}}}

\def\@slothm#1[#2]#3{\expandafter\@ifdefinable\csname #1\endcsname
  {\global\@namedef{the#1}{\@nameuse{the#2}}%
\global\@namedef{#1}{\@slthm{#2}{#3}}%
\global\@namedef{end#1}{\@endsltheorem}}}

\def\@slthm#1#2{\refstepcounter
    {#1}\@ifnextchar[{\@slythm{#1}{#2}}{\@slxthm{#1}{#2}}}

\def\@slxthm#1#2{\@beginsltheorem{#2}{\csname the#1\endcsname}\ignorespaces}
\def\@slythm#1#2[#3]{\@opargbeginsltheorem{#2}{\csname
       the#1\endcsname}{#3}\ignorespaces}

\def\@slthmcounter#1{.\noexpand\arabic{#1}}
\def\@slthmcountersep{}
\def\@beginsltheorem#1#2{\sl \trivlist
      \item[\hskip \labelsep{\bf #1\ #2\thmslcounterend}]}
\def\@opargbeginsltheorem#1#2#3{\sl \trivlist
      \item[\hskip \labelsep{\bf #1\ #2\ (#3)\thmslcounterend}]}
\def\@endsltheorem{\endtrivlist}
\def\thmslcounterend{\hskip 0em\relax}
\def\newslwntheorem#1#2{\expandafter\@ifdefinable\csname #1\endcsname%
\global\@namedef{#1}{\@slwnthm{#1}{#2}}%
\global\@namedef{end#1}{\@endslwntheorem}}

\makeatother

\newsltheorem{theorem}{Theorem}[section]
\newsltheorem{proposition}{Proposition}[section]
\newsltheorem{lemma}{Lemma}[section]
\newsltheorem{corollary}{Corollary}[section]

\newrmtheorem{remark}{Remark}[section]
\newrmtheorem{problem}{Problem}[section]
\newrmtheorem{example}{Example}[section]

\theoremstyle{definition}

\newtheorem{algorithm}{Algorithm}

\numberwithin{equation}{section}

\begin{document}

\providecommand{\keywords}[1]
{
  \small	
  \textsl{Keywords:} #1
}

\providecommand{\ams}[1]
{
  \small	
  \textsl{AMS subject classification:} #1
}

\title{Synchronous deautoconvolution algorithm \\
for discrete-time positive signals \\
via $\ii$-divergence approximation}%
\author{Lorenzo Finesso\thanks{Lorenzo Finesso is with Institute of Electronics, Information Engineering and Telecommunications,
National Research Council, CNR-IEIIT, Padova; email: {\tt lorenzo.finesso@unipd.it}}
 \and
        Peter Spreij\thanks{Peter Spreij is with the Korteweg-de Vries Institute for Mathematics,
Universiteit van Amsterdam and with the  Institute for Mathematics, Astrophysics and Particle Physics, Radboud University, Nijmegen; e-mail: {\tt spreij@uva.nl}}
\thanks{The material in this paper was not presented at any conference.}
}

\maketitle
\begin{abstract}
We pose the problem of the optimal approximation of a given nonnegative signal $y_t$ with the scalar autoconvolution $(x*x)_t$ of a nonnegative signal $x_t$, where $x_t$ and $y_t$ are signals of equal length. The $\ii$-divergence has been adopted as optimality criterion, being well suited to incorporate nonnegativity constraints. To find a minimizer we derive an iterative descent algorithm of the alternating minimization type. The algorithm is based on the lifting of the original problem to a larger space, a relaxation technique developed by Csisz\'ar and Tusn\'ady in [Statistics \& Decisions (S1) (1984), 205--237] which, in the present context, requires the solution of a hard partial minimization problem. We study the asymptotic behavior of the algorithm  exploiting the optimality properties of the partial minimization problems and prove, among other results, that its limit points are Kuhn-Tucker points of the original minimization problem. Numerical experiments illustrate the results.


\ams{93B30, 94A17}

\end{abstract}

\section{Introduction}

The focus of the paper is on the inverse problem of positive autoconvolution. Given the data sequence of nonnegative observations $y_t$, $t$ in an appropriate time domain, find a nonnegative signal $x_t$ such that the autoconvolutions $(x*x)_t:=\sum_{i=0}^t x_ix_{t-i}$ form an optimal approximation of the $y_t$. We call such $x_t$ the positive \emph{deautoconvolution} of $y_t$.
The problem of finding the deautoconvolution can be posed in two different setups.
The first is to approximate the data sequence $y_t$, $t=0,\ldots, n$, with the autoconvolutions $(x*x)_t$ of a finite support signal $x_t$, $t=0,\ldots, m$. Since the support of $(x*x)_t$ is $[0, 2m]$, in this setup $n=2m$, i.e.\ the support of the deautoconvolution is half of the support of $y_t$, hence we call this setup the half support case.
An algorithm for finding the deautoconvolution $x_t$ in the half support case has been proposed, and its properties have been discussed, in the companion paper \cite{fs2021inverse}. The second approach is to approximate the data sequence $y_t$, $t=0,\ldots n$ with the initial segment of the autoconvolution $(x*x)_t$, $t=0,\dots, n$ of a signal $x_t$ of long, unspecified length. In this case the problem is to find the initial segment of the deautoconvolution, i.e.\ the signal $x_t$, $t=0,\ldots, n$. In this approach the data sequence $y_t$ and the deautoconvolution $x_t$ have common support, $t=0,\ldots, n$, hence this is called the full support or \emph{synchronous} case. The present paper deals with the synchronous setup. Despite obvious similarities between the two setups, it turns out that the algorithm to construct the deautoconvolution is considerably more involved in the synchronous setup, and the results strongly differ between the two cases. A similar situation, with two alternative setups, arises also for the simpler problem of nonnegative deconvolution, in which one approximates optimally the given data sequence $y_t$ with the convolution $(h*x)_t$, where $h_t$ is known and $x_t$ is to be determined ($y_t$, $h_t$ and $x_t$ being all nonnegative). Compare with \cite{finessospreij2015ieeeit}, \cite{finessospreij2019automatica}.

Extensive work has been dedicated to the inverse problem of autoconvolution for functions on the real line, mostly emphasizing the functional analytic aspects and motivating its interest in a variety of applications in physics and engineering. We briefly mention in this context the references \cite{dose1981inversion}, \cite{martinez1979global}, \cite{hofmann2006determination}, \cite{douglas2014autoconvolution}.
The main differences between the cited literature for deconvolution of functions and this paper are that we consider approximation problems, rather than looking for exact solutions which exist only exceptionally, and that our (time) domain is discrete rather than the real line. Moreover, the nonnegativity constraint, that we impose on all signals, is a crucial feature of the present work.
Some earlier contributions share, at least in part, our point of view, e.g.\ the papers~\cite{choilanterman2005}, \cite{choilantermanraich2006}, contain an algorithm of the same type as ours, but valid in the half support case. For more detailed comments on ~\cite{choilanterman2005}, \cite{choilantermanraich2006}, see \cite{fs2021inverse}.

The purpose of this paper is threefold. First we pose the problem of time domain approximation of a given discrete time nonnegative data sequence with finite autoconvolutions. Following the choice made in other optimization problems, we opt for the $\ii$-divergence as a criterion, which as argued in~\cite{cs1991} (see also \cite{snyderetal1992}), is the natural choice under nonnegativity constraints. We provide a result on the existence of the minimizer of the approximation criterion.
Then we propose an iterative algorithm, in the spirit of \cite{ct1984}, to find the best approximation, and present a number of results on its behavior. Our approach -- alternating minimization -- is based on earlier works like~\cite{fs2006} on nonnegative matrix factorization, \cite{finessospreij2015ieeeit}, \cite{finessospreij2019automatica}, on  \emph{linear} convolutional problems, in contrast to the inherently nonlinear autoconvolution.

The inherent nonconvexity, and nonlinearity of the problem make the analysis of the asymptotic behavior difficult, even more so because of the existence of several local minima of the objective function and possibly identification issues as well.
The main result in this respect is Theorem~\ref{thm:kt} which states that limit points of the algorithm satisfy the Kuhn-Tucker optimality conditions, which is, in formulation, similar to the corresponding result in \cite{fs2021inverse}. In the contest of nonlinear non convex optimization problems, like the present one, to the best of our knowledge there are no asymptotic results that improve on our Theorem~\ref{thm:kt}.

A brief summary of the paper follows. In Section~\ref{section:problem} we state the full support minimization problem, show the existence of a solution and give some of its properties. In Section~\ref{section:lift} the original problem is lifted into a higher dimensional setting, thus making it amenable to alternating minimization. The optimality properties (Pythagoras rules) of the ensuing partial minimization problems are discussed here. It turns out that one of these partial minimization problems has a very complex solution, whose precise representation will be derived in the appendix. In Section~\ref{section:algorithm} we derive the iterative minimization algorithm combining the solutions of the partial minimizations, and analyse its properties. In particular we show that limit points of the algorithm are Kuhn-Tucker points of  the original optimization problem. In the concluding Section~\ref{section:numerics} we present numerical experiments that shed light on the theoretical results concerning the behaviour of the algorithm and also highlight some unexpected phenomena.

\section{Problem statement and preliminaries}\label{section:problem}

In the paper we consider real valued signals $x: \mathbb Z \rightarrow \mathbb R$, maps $i \mapsto x_i$, that vanish for $i<0$, i.e., $x_i=0$ for $i<0$. The \emph{support} of $x$ is the discrete time interval, $[0, n]$, where $n=\inf\{\,k:\,\, x_i=0,\,\,\, \text{for $i> k$}\,\}$, if the infimum is finite (and then a minimum), or $[0, \infty)$ otherwise. The autoconvolution of $x$ is the signal $x* x$, vanishing for $i<0$, and satisfying,
\begin{equation}\label{eq:xconv}
(x* x)_i =\sum_{j=-\infty}^\infty x_{i-j}x_j =
                 \sum_{j=0}^i x_{i-j}x_j\,,  \qquad i \ge 0.
\end{equation}
Notice that if the support of $x$ is finite $[0, n]$, the support of $x* x$ is $[0,2n]$. In this case, when computing $(x * x)_i$ for $i>n$, the summation in \eqref{eq:xconv} has non zero addends only in the range $i-n\le j \le n$, as $x_{i-j}=0$ and $x_j=0$ for $i-j>n$ and $j>n$ respectively. If the signal $x$ is nonnegative, i.e.\ $x_i \ge0$ for all $i\in \mathbb Z$,  the autoconvolution (\ref{eq:xconv}) is too.

The paper addresses the inverse problem of autoconvolution. Given a finite \emph{nonnegative} data sequence
$$
y=(y_0,\dots , y_n)\,,
$$
find a \emph{nonnegative} signal $x$ whose autoconvolution $x*x$ best approximates $y$. Since the signals involved are nonnegative, it is natural to adopt the $\ii$-divergence as the approximation criterion, see~\cite{c1975,cs1991}. Recall that the $\ii$-divergence between two nonnegative vectors $u$ and $v$ of equal length is
\[
\ii(u,v)=\sum_i \Big(u_i\log\frac{u_i}{v_i}-u_i+v_i\Big)\,,
\]
if $u_i=0$ whenever $v_i=0$, and $\ii(u,v)=\infty$ if there exists an index $i$ with $u_i>0$ and $v_i=0$. It is known that $\ii(u,v)\geq 0$, with equality iff $u=v$.

The approximation problem depends on the constraints imposed on the support of $x$. In Problem~\ref{problem} below, the signal $x$ is assumed to have the same length of the data sequence, $x=(x_0,\dots, x_n)$ and produces the approximation problem specified below, where we write $x* x\in\rr^{n+1}$ for the restriction to $[0, n]$ of the convolution defined in~\eqref{eq:xconv}. This version of the problem arises naturally when the observed data sequence is the initial segment, $y=(y_0,\ldots,y_n)$, of a signal $y$ of long unspecified length, modeled with the initial segment of the autoconvolution $x*x$ of a signal $x$ of long unspecified length.
\begin{problem}\label{problem}
Given $y\in\rr^{n+1}_+$ minimize, over $x\in\rr^{n+1}_+=[0,\infty)^{n+1}$,
\begin{equation} \label{cost-full}
\ii=\ii(x):=\ii(y||x* x)=\sum_{i=0}^n\Big(y_i\log\frac{y_i}{(x* x)_i}-y_i+(x* x)_i\Big)\,.
\end{equation}
\end{problem}
The objective function \eqref{cost-full} is nonconvex and nonlinear in $x$, the existence of a minimizer is therefore not immediately clear. Our first result settles in the affirmative the question of the existence. The issue of uniqueness remains open, but we have evidence of the existence of multiple local minima of $\ii(x)$. See Section~\ref{section:numerics} for numerical examples.
\begin{proposition}\label{proposition:exist}
Let $y_0>0$, then Problem~\ref{problem} admits a  solution.
\end{proposition}

\begin{proof} Deferred to Section~\ref{section:algorithm} as we need properties of the algorithm of that section.
\end{proof}
For later reference we compute the components of the gradient $\nabla\ii(x)$ of the objective function (\ref{cost-full}).
\begin{equation}\label{eq:iprime}
\nabla\!_k\ii(x):=\frac{\partial \ii(x)}{\partial x_k}=-2\sum_{j=0}^{n-k}\,\Big(\frac{y_{j+k}}{(x*x)_{j+k}} - 1\Big)\,x_{j}.
\end{equation}
We will also often use the notations
\begin{align}
\widehat y_i &:= (x*x)_i, \label{eq:yhat}\\[2pt]
\rho_i &:= \frac{y_i}{\widehat y_i}, \label{eq:rho}
\end{align}
and as a result we have the following alternative ways of writing Equation~\eqref{eq:iprime},
\begin{align*}
\nabla\!_k\ii(x) & = -2\sum_{j=0}^{n-k}\big(\rho_{k+j}-1\big)x_{j}=  -2\sum_{j=k}^n\big(\rho_j-1\big)x_{j-k},
\end{align*}
and then also
\begin{equation}\label{eq:nablarho}
\sum_{j=k}^n\rho_jx_{j-k}  = -\half\nabla\!_k\ii(x)+\sum_{j=k}^n x_{j-k}.
\end{equation}
The expressions in \eqref{eq:iprime} are highly nonlinear in $x$ and solving the first order optimality conditions $\nabla \ii(x)=0$ to find the stationary points of \eqref{cost-full}, will not result in closed form solutions except in trivial cases. This observation calls for a numerical approach to the optimization, which we will present in Section~\ref{section:algorithm}.
First a useful property of minimizers of $\ii$.
\begin{proposition}\label{prop:sum}
Minimizers $x^\star$ are such that $\sum_{j=0}^n(x^\star*x^\star)_j=\sum_{j=0}^ny_j$.
\end{proposition}
\begin{proof}
Suppose that $x^\star$ is a minimizer of $\ii(x)$ and
let $f(\alpha)=\ii(\alpha x^\star)$, $\alpha>0$. Then $f'(1)=0$. A direct computation of $f(\alpha)$ gives after differentiation $f'(\alpha)=-\frac{2}{\alpha}\sum_{i=0}^ny_i+2\alpha\sum_{i=0}(x^\star*x^\star)_i$. Hence $f'(1)=0$ yields the result.
\end{proof}
We now give some examples, where the minimizer can be explicitly calculated. These are exceptional cases.
\begin{example}\label{example:n=1}
Let $n=1$. We will look at four cases.

(i) $y_0=y_1=0$. Here $\ii(x)=x_0^2+2x_0x_1$ and we observe that $\ii$ is minimized for $x_0=0$ and arbitrary $x_1$. The minimum is not uniquely attained.

(ii) $y_0=0<y_1$. Then $\ii(x)=y_1\log\frac{y_1}{2x_0x_1}-y_1+x_0^2+2x_0x_1$. Let $x_0,x_1$ be such that $2x_0x_1=y_1$, then $\ii(x)=x_0^2$ and we see that the infimum of $\ii(x)$ is not attained.

(iii) $y_1= 0<y_0$. Then $\ii(x)=y_0\log\frac{y_0}{x_0^2}-y_0+x_0^2+2x_0x_1$.
Here one shows that the minimium is uniquely attained for a boundary point $x$ with $x_0=\sqrt{y_0}$ and $x_1=0$, which gives $\ii(x)=0$. At the minimum we have that the gradient $\nabla \ii(x)=(0,2x_0)$.

(iv) $y_0,y_1>0$. The minimum is uniquely attained at an interior point $x$ of the domain with $x_0=\sqrt{y_0}$ and $x_1=\frac{y_1}{2\sqrt{y_0}}$. At the minimum we have $\nabla \ii(x)=(0,0)$.

The message of this example is that $y_0=0$ (cases (i) and (ii)) displays undesired properties when searching for a minimum, and that minimizers at the boundary of the domain occur.
\end{example}

\begin{example}\label{example:n=2}
Let $n=2$ and $y_0>0$. Then it depends on the specific values of the $y_k$ whether the problem has a solution in the interior of the domain or on its boundary. The former happens if $y_1^2<4y_0y_2$, which gives the minimizing
\begin{align*}
x_0 & =\sqrt{y_0} \\
x_1 & =\frac{y_1}{2\sqrt{y_0}} \\
x_2 & =\frac{y_2-\frac{y_1^2}{4y_0}}{2\sqrt{y_0}}\,,
\end{align*}
and then the divergence equals zero. If $y_1^2
\geq 4y_0y_2$, then the solution is quite different. In this case the minimizers become
\begin{align*}
x_0 & =\frac{y_0+\frac{y_1}{2}}{\sqrt{y_0+y_1+y_2}} \\
x_1 & =\frac{y_2+\frac{y_1}{2}}{\sqrt{y_0+y_1+y_2}} \\
x_2 & =0\,.
\end{align*}
Of course the two formulas coincide in the boundary case $y_1^2=4y_0y_2$, which implies $y_0+y_1+y_2=(\sqrt{y_0}+\sqrt{y_2})^2$, as they should. Note that in the second case the partial derivatives $\frac{\partial \ii}{\partial x_j}$ are zero at the minimizer for $j=0,1$, whereas
\[
\frac{\partial \ii}{\partial x_2}
=\frac{\big(y_0+\frac{y_1}{2}\big)\big(\frac{y_1^2}{4}-y_0y_2\big)}{\big(y_2+\frac{y_1}{2}\big)^2\sqrt{y_0+y_1+y_2}}\geq 0\,.
\]
\end{example}
The next proposition shows that solutions to Problem~\ref{problem} at which the gradient vanishes are exceptions rather than the rule.
\begin{proposition}
Assume $y_0>0$. A point $x\in\rr^{n+1}_+$ is a solution at which the gradient $\nabla\ii$ vanishes iff it gives a perfect match with the $y_k$, i.e.\ $(x*x)_k=y_k$ for all $k=0,\ldots,n$.
\end{proposition}

\begin{proof}
Consider Equation~\eqref{eq:iprime}. From this equation one immediately sees that an exact match produces a solution where the gradient vanishes.
To show the converse, note first that $x_0>0$ and let $\nabla\ii(x)=0$. Then, by virtue of~\eqref{eq:iprime}, for all $k=0,\ldots,n$ one has $\sum_{j=0}^{n-k}\big(\frac{y_{j+k}}{(x*x)_{j+k}}-1\big)x_{j}=0$. For $k=n$ this equation reads $(\frac{y_{n}}{(x*x)_{n}}-1)x_0=0$ and hence $y_n=(x*x)_{n}$. For $k=n-1$, the equation becomes $(\frac{y_{n-1}}{(x*x)_{n-1}}-1)x_{0}+(\frac{y_{n}}{(x*x)_{n}}-1)x_{1}=0$, from which we can now deduce $y_{n-1}=(x*x)_{n-1}$. Counting $k$ down to zero yields the assertion.
\end{proof}

\section{Lifting}\label{section:lift}

Following the approach of the companion paper \cite{fs2021inverse} we recast Problem~\ref{problem} as a double minimization problem by lifting it into a larger space, a necessary step to derive the minimization algorithm.
Related work with similar techniques is~\cite{fs2006} to analyse a nonnegative matrix factorization problem, and closer to the present paper \cite{finessospreij2015ieeeit}, \cite{finessospreij2019automatica}, noting that the latter references treat linear convolutional problems, whereas the autoconvolution is inherently nonlinear.

The ambient spaces for the lifted problem are the subsets $\bby$ and $\bbw$, defined below, of the set of matrices $\rr^{(n+1)\times (n+1)}_+$,
$$
\bby := \Big\{\,\by  : \,\,\, \textstyle{\sum_j}\by_{ij} = y_i\,, \quad\text{and}\quad \by_{ij}=0,\,\, \text{for}\,\, i < j\, \Big\}\,,
$$
where $y=(y_0,\dots, y_{n})\in\mathbb R^{n+1}_+$ is the given data vector, and
$$
\bbw := \Big\{\,\bw  : \,\,\, \bw_{ij} = x_{i-j}x_j\,,\,\,\text{if}\,\, j\leq i\,,\quad\text{and}\quad \bw_{ij}=0\,,\,\,\text{for}\,\, i < j \,\Big\}\,,
$$
where $x=(x_0,\ldots x_n)$ are nonnegative parameters. The structure of the matrices in $\bby$ and $\bbw$ is shown below for $n=6$,

\medskip
\small
$$
\by=\begin{pmatrix}
\by_{00} & 0 & 0 & 0 & 0 & 0 & 0\\
\by_{10} & \by_{11}& 0 & 0 & 0 & 0 & 0\\
\by_{20} & \by_{21}& \by_{22} & 0 & 0 & 0 & 0\\
\by_{30} & \by_{31}& \by_{32} & \by_{33} & 0 & 0 & 0 \\
\by_{40} & \by_{41} & \by_{42} & \by_{43} & \by_{44} & 0 & 0 \\
\by_{50} & \by_{51} & \by_{52} & \by_{53} & \by_{54} & \by_{55} & 0 \\
\by_{60} & \by_{61} & \by_{62} & \by_{63} & \by_{64} & \by_{65} & \by_{66}
\end{pmatrix},
$$

\bigskip
$$
\bw=\begin{pmatrix}
x_0x_0 & 0      & 0      & 0      & 0      & 0      & 0\\
x_1x_0 & x_0x_1 & 0      & 0      & 0      & 0      & 0\\
x_2x_0 & x_1x_1 & x_0x_2 & 0      & 0      & 0      & 0\\
x_3x_0 & x_2x_1 & x_1x_2 & x_0x_3 & 0      & 0      & 0 \\
x_4x_0 & x_3x_1 & x_2x_2 & x_1x_3 & x_0x_4 & 0      & 0 \\
x_5x_0 & x_4x_1 & x_3x_2 & x_2x_3 & x_1x_4 & x_0x_5 & 0 \\
x_6x_0 & x_5x_1 & x_4x_2 & x_3x_3 & x_2x_4 & x_1x_5 & x_0x_6
\end{pmatrix}\,.
$$
\medskip\\
\normalsize
The interpretation is as follows. The matrices $\by\in\bby$ and $\bw\in\bbw$ have common support on  (sub)diagonals. The row marginal (i.e.\ the column vector of row sums) of any $\by\in \bby$ coincides with the given data vector $y$. The elements of the $\bw$ matrices factorize and their row marginals are the autoconvolutions $(x*x)_i$, $i=0,\ldots,n$. Note also that the elements $\bw\in\bbw$ exhibit the symmetry $\bw_{ij}=\bw_{i,i-j}$, $i\geq j$.
\medskip\\
We introduce two partial minimization problems over the subsets $\bby$ and $\bbw$. Recall that the $\ii$-divergence between two nonnegative matrices of the same sizes $M, N\in\rr^{p\times q}_+$ is defined as
\[
\ii(M||N) := \sum_{i,j} \Big( M_{ij} \log \frac{M_{ij}}{N_{ij}} - M_{ij} + N_{ij}\Big)\,.
\]

\begin{problem}\label{problemy}
Minimize
$\ii(\by||\bw)$ over $\by\in\bby$ for given $\bw\in\bbw$.
\end{problem}

\begin{problem}\label{problemw}
Minimize
$\ii(\by||\bw)$ over $\bw\in\bbw$ for given $\by\in\bby$.
\end{problem}
The relation between the original Problem~\ref{problem} and the partial minimization Problems~\ref{problemy} and \ref{problemw} is as follows.

\begin{proposition}
It holds that
\[
\min_{x\in\rr^{n+1}_+}\ii(x)=\min_{\by\in\bby}\min_{\bw\in\bbw}\ii(\by||\bw),
\]
where in all three cases the minima are attained.
\end{proposition}

\begin{proof}
The proof is almost the same as that of \cite[Proposition~3.8]{fs2021inverse}.
\end{proof}
The partial minimization problems are instrumental in deriving the algorithm in Section~\ref{section:algorithm} for which one needs that these two problems have solutions in closed form.
For Problem~\ref{problemy} it can easily be given, Lemma~\ref{lemmay}. On the other hand, the solution to Problem~\ref{problemw}, see Lemma~\ref{lemmaw} that relies on results in Appendix~\ref{app:equation}, turns out to be substantially more complicated and therefore radically differs from the solution to the companion Problem~3.2 in \cite{fs2021inverse}. The remainder of this section is devoted to properties of the solutions to the partial minimization problems.
\begin{lemma}\label{lemmay}
Problem~\ref{problemy} has the explicit minimizer $\by^\star\in\bby$ given by
\begin{align*}
\by^\star_{ij} & = \frac{y_i}{\sum_j\bw_{ij}}\bw_{ij} \\
& = \frac{x_jx_{i-j}}{\sum_jx_jx_{i-j}}y_i=\frac{x_jx_{i-j}}{(x*x)_i}y_i\,,\label{eq:ystar}
\end{align*}
for $\bw\in\bbw$ and $i\geq j$. Moreover, $\ii(\by^\star||\bw)=\ii(y||x*x)$ and the \emph{Pythagorean identity}
\[
\ii(\by||\bw)=\ii(\by||\by^\star)+\ii(\by^\star||\bw)\,,
\]
holds for any $\by\in\bby$.
\end{lemma}

\begin{proof}
The proof can be given by a direct computation that is almost verbatim the same as the proof of Lemma~3.3 in \cite{fs2021inverse}. Alternatively, the Pythagorean identity also follows from \cite[Theorem~3.2]{csiszarshields}, as the set $\bby$ is a linear set in the terminology of \cite{csiszarshields}.
\end{proof}

\begin{remark}\label{remark:ysym}
The optimal $\by^\star$ in Lemma~\ref{lemmay} exhibits the  symmetry $\by^\star_{ij}=\by^\star_{i,i-j}$ for all $i\geq j$, which is the same symmetry as enjoyed by the $\bw\in\bbw$.
\end{remark}

\begin{lemma}\label{lemmaw}
Consider Problem~\ref{problemw} for given $\by\in\bby$ and let for $j=0,\ldots,n$,
\[
\widetilde\by_j=\sum_{i=j}^n\by_{ij}+\sum_{i=j}^n\by_{i,i-j}\,.
\]
The minimizing $x_i^\star$ uniquely exist and are the solutions to the nonlinear system of equations
\begin{equation}\label{eq:xstar}
2x_j\sum_{i=0}^{n-j}x_i=\widetilde\by_j\,,\qquad j=0,\ldots,n.
\end{equation}
The Pythagorean identity holds for the solutions, i.e.\ for every $\bw\in\bbw$ it holds that
\begin{equation}\label{eq:pythwstar}
\ii(\by||\bw)=\ii(\by||\bw^\star)+\ii(\bw^\star||\bw)\,,
\end{equation}
where $\bw^\star$ is the element of $\bbw$ corresponding to the $x_i^\star$.
Furthermore, it holds that
\[
\sum_{i=0}^n(x^\star*x^\star)_i=\sum_{ij}\by_{ij}\,.
\]
\end{lemma}

\begin{proof}
We compute
\[
\ii(\by||\bw)=\sum_{ij}\by_{ij}\log\frac{\by_{ij}}{\bw_{ij}}-\sum_{ij}\by_{ij}+\sum_{ij}\bw_{ij}\,.
\]
It follows that we have to minimize
\begin{align*}
F(\bw) & = -\sum_{ij}\by_{ij}\log\bw_{ij}+\sum_{ij}\bw_{ij} \\
& = -\sum_{i=0}^n\sum_{j=0}^i\by_{ij}\log (x_jx_{i-j})+\sum_{i=0}^n\sum_{j=0}^ix_jx_{i-j} \\
& = -\sum_{i=0}^n\sum_{j=0}^i\by_{ij}\log x_j-\sum_{i=0}^n\sum_{j=0}^i\by_{i,i-j}\log x_j+\sum_{i=0}^n\sum_{j=0}^ix_jx_{i-j} \\
& = -\sum_{j=0}^n\Big(\sum_{i=j}^n\by_{ij}\Big)\log x_j-\sum_{j=0}^n\Big(\sum_{i=j}^n\by_{i,i-j}\Big)\log x_j+\sum_{i=0}^n\sum_{j=0}^ix_jx_{i-j} \\
& = -\sum_{j=0}^n\widetilde\by_j\log x_j+\sum_{i=0}^n\sum_{j=0}^ix_jx_{i-j}\,,
\end{align*}
We compute the gradient of $F$,
\[
\frac{\partial F}{\partial x_j}= -\frac{\widetilde \by_j}{x_j}+2\sum_{i=0}^{n-j}x_i\,.
\]
Setting $\frac{\partial F}{\partial x_j}=0$ for all $j$ gives equations \eqref{eq:xstar}.

Next we show that this system of equations has a unique solution. Let
\begin{equation}
r_j := \half\widetilde\by_j\,, \label{eq:ry}
\end{equation}
and $k=\lfloor \frac{n}{2}\rfloor$.
Observe that the $r_j$ in \eqref{eq:ry} are nonnegative by the assumption on the $\by_{ij}$ in $\bby$.
In the new notation with the $r_j$ we have to solve
\begin{equation}\label{eq:xstarr}
x_j\sum_{i=0}^{n-j}x_i=r_j\,,\qquad j=0,\ldots,n.
\end{equation}
To show that a solution to \eqref{eq:xstarr} exists, we have to verify Condition~\eqref{eq:s2} of Lemma~\ref{lemma:R} for the $r_j$ given by \eqref{eq:ry}.
The result then follows by application of Propositions~\ref{prop:neven} or~\ref{prop:nodd}. We write
\[
S^2=\Bigg(\sum_{j=0}^k\sum_{i=j}^n\by_{ij} - \sum_{j=k+1}^n\sum_{i=j}^n\by_{i,i-j}\Bigg) + \Bigg(\sum_{j=0}^k\sum_{i=j}^n\by_{i,i-j} - \sum_{j=k+1}^n\sum_{i=j}^n\by_{ij}\Bigg).
\]
Consider the first term in parentheses in $S^2$,
\begin{equation}\label{eq:first}
\sum_{j=0}^k\sum_{i=j}^n\by_{ij} - \sum_{j=k+1}^n\sum_{i=j}^n\by_{i,i-j}\,,
\end{equation}
which we rewrite by splitting the first sum as
\begin{equation}\label{eq:split1}
\sum_{j=0}^k\sum_{i=j}^{j+k}\by_{ij} + \sum_{j=0}^k\sum_{i=j+k+1}^n\by_{ij} - \sum_{j=k+1}^n\sum_{i=j}^n\by_{i,i-j}\,.
\end{equation}
We shall now work on the last two sums in \eqref{eq:split1}. We develop, first by interchanging the summation order and then using a change of variable for the second summation,
\begin{align*}
\sum_{j=0}^k\sum_{i=j+k+1}^n\by_{ij} - \sum_{j=k+1}^n\sum_{i=j}^n\by_{i,i-j}
& = \sum_{i=k+1}^n\sum_{j=0}^{i-(k+1)}\by_{ij} - \sum_{i=k+1}^n\sum_{j=k+1}^i\by_{i,i-j} \\
& = \sum_{i=k+1}^n\sum_{j=0}^{i-(k+1)}\by_{ij} - \sum_{i=k+1}^n\sum_{\ell=0}^{i-(k+1)}\by_{i,\ell} \\
& = 0.
\end{align*}
It follows that the quantity in \eqref{eq:first}, the first term in parentheses in $S^2$,  is nonnegative. The proof for the second term in parentheses in  $S^2$ is similar.

To prove the Pythagorean identity \eqref{eq:pythwstar}, we compute both the quantities $\ii(\by||\bw)-\ii(\by||\bw^\star)$ and $\ii(\bw^\star||\bw)$ and show that they are the same. We put $\widetilde\ii(\by||\bw)=\sum_i\sum_{j\leq i}\by_{ij}\log\frac{\by_{ij}}{\bw_{ij}}$ and a similar convention for the other modified divergences. It is then sufficient to show that $\widetilde\ii(\by||\bw)-\widetilde\ii(\by||\bw^\star)$ and $\widetilde\ii(\bw^\star||\bw)$ are the same. Recall $\bw_{ij}=x_jx_{i-j}$ for $i\geq j$ and zero otherwise. For the elements $\bw^\star_{ij}$ one has similar expressions. A direct computation now gives
\begin{align*}
\widetilde\ii(\by||\bw)-\widetilde\ii(\by||\bw^\star)
& = \sum_i\sum_{j\leq i} \by_{ij}\log\frac{x_j^\star x^\star_{i-j}}{x_jx_{i-j}} \\
& = \sum_i\sum_{j\leq i} \by_{ij}\log\frac{x_j^\star}{x_j}  + \sum_i\sum_{j\leq i} \by_{ij}\log\frac{x^\star_{i-j}}{x_{i-j}} \\
& = \sum_i\sum_{j\leq i} \by_{ij}\log\frac{x_j^\star}{x_j}  + \sum_i\sum_{l\leq i} \by_{i,i-l}\log\frac{x^\star_{l}}{x_{l}} \\
& = \sum_j\sum_{i\geq j}\by_{ij}\log\frac{x_j^\star}{x_j} + \sum_l\sum_{i\geq l} \by_{i,i-l}\log\frac{x^\star_{l}}{x_{l}} \\
& = \sum_j\Big(\sum_{i\geq j}\by_{ij} + \sum_{i\geq j} \by_{i,i-j}\Big)\log\frac{x_j^\star}{x_j} \\
& = \sum_j\widetilde\by_j\log\frac{x_j^\star}{x_j}\,.
\end{align*}
Similar computations lead to
\[
\widetilde\ii(\bw^\star||\bw) = \sum_j\Big(2x_j^\star\sum_{i=j}^{n}x^\star_{i-l}\Big)\log\frac{x_j^\star}{x_j}\,.
\]
Since the $x_j^\star$ solve the system of equations~\eqref{eq:xstar}, the term in parentheses equals $\widetilde\by_j$ and we are done with proving the Pythagorean identity.
As the optimal $x^\star$ satisfies \eqref{eq:xstar}, we can sum the equalities over all $j$ to get by swapping the summation order $2\sum_{i=0}^n(x^\star*x^\star)=\sum_{ij}\tilde\by_{ij}$. But $\sum_{ij}\tilde\by_{ij}$ just equals $2\sum_{ij}\by_{ij}$ by the special structure of the matrix $\by$. The last claim follows.
\end{proof}

\begin{remark}\label{remark:symm2}
Under the symmetry condition on $\by$ as in Remark~\ref{remark:ysym}, Equation~\eqref{eq:xstar} reduces to
\[
x_j\sum_{i=0}^{n-j}x_i=\sum_{i=j}^n\by_{ij}
=\sum_{i=j}^n\by_{i,i-j}\,,\qquad j=0,\ldots,n.
\]
This means that, for every $j$, the sum of the $j$-th column of the matrix $\bw$ equals the sum of the $j$-th column of the matrix $\by$ and, by symmetry, that the sum of the corresponding subdiagonal of the matrix $\bw$ equals that of the sum of  its counterpart of the matrix $\by$.

\end{remark}

\section{The algorithm}\label{section:algorithm}

This section contains the key results of the paper. We derive an algorithm  producing iterates $x^t\in\rr^{n+1}_+$, $t\geq 0$, aiming at approximating a solution to Problem~\ref{problem}. The algorithm is of alternating minimization type, using the partial minimization problems of Section~\ref{section:lift}. This type of algorithms originates with the seminal paper \cite{ct1984} by Csisz\'ar and Tusn\'ady.
Results of Appendix~\ref{app:equation} are used throughout this section.

Start with a vector $x^0=(x^0_0,\ldots,x^0_n)$  with positive elements. An alternating concatenation of the solutions to Problems~\ref{problemy} and~\ref{problemw} yields the following. All relevant quantities of the algorithm have an index $t$ at the $t$-th step. Suppose the $t$-iterate $x^t$ is given, and then also $\bw^t$. According to Lemma~\ref{lemmay}, with $\bw^t$ as input, and recalling (\ref{eq:yhat}), (\ref{eq:rho}),
the optimal $\by^t$ is given by
\[
\by^t_{ij}=\frac{x^t_jx^t_{i-j}}{(x^t*x^t)_i}y_i=x^t_jx^t_{i-j}\rho^t_i\,.
\]
In the next step, one uses the $\by^t$ as input to compute according to Lemma~\ref{lemmaw} the optimal $x^{t+1}$ and $\bw^{t+1}$. So, one has to solve the analog of \eqref{eq:xstar}, i.e.\ finding the $x^{t+1}_j$ as the solutions to
\[
2x^{t+1}_j\sum_{i=0}^{n-j}x^{t+1}_i=\widetilde\by^t_j\,,\quad j=0,\ldots, n,
\]
or, in view of Remark~\ref{remark:symm2}, the solutions to
\begin{equation}\label{eq:xt}
x^{t+1}_j\sum_{i=0}^{n-j}x^{t+1}_i=r^t_j:=\sum_{i=j}^n\by^t_{ij}\,,\quad j=0,\ldots, n.
\end{equation}
The $x^{t+1}_j$  can be taken as the $x^\star_j$ in Proposition~\ref{prop:neven} for even $n$, or Proposition~\ref{prop:nodd} for odd $n$.  Although explicit expressions are complicated, the above procedure is amenable to a practically implementable algorithm.
The following algorithm  is the result of this procedure.

\begin{algorithm}\label{algo}
The data are given as $(y_0,\ldots,y_n)$.
Initiate the algorithm at a point $x^0\in (0,\infty)^{n+1}$ and let
$x^t=(x^t_0,\ldots,x^t_n)$ be the values at the $t$-th iteration of the algorithm. To compute the values $x^{t+1}_i$ at the $t+1$-th iteration:
\begin{enumerate}
\item
Compute the $\widehat y^t_j$ and $\rho^t_j$ as (see (\ref{eq:yhat}), (\ref{eq:rho}))
\[
\widehat y^t_j=(x^t*x^t)_j\,,\qquad \rho^t_j=\frac{y_j}{\widehat y^t_j}\,.
\]
\item
Compute $r^t_j$ from the $x^t_j$ according to
\[
r^t_j=x^t_j\sum_{i= j}^n x^t_{i-j}\rho^t_i=x^t_j\sum_{i=0}^{n-j}x^t_{i}\rho^t_{i+j}\,.
\]
\item
Compute the $B^t_j$ and $E^t_j$ from the $r^t_j$ according to \eqref{eq:bj} and \eqref{eq:ej}.
\item
Use Proposition~\ref{prop:neven} ($n$ even) or Proposition~\ref{prop:nodd} ($n$ odd) to compute the $x^{t+1}_m$ as the $x^\star_m$ in those propositions from the $B^t_j$ and $E^t_j$.
\end{enumerate}
\end{algorithm}
\noindent
Here we present an alternative description of the algorithm,  only for the case where $n$ is even, $n=2k$ to get more insight in the way Algorithm~\ref{algo} is structured. This also forms the basis of a practical implementation. Suppose all $x^t_j$ and related quantities are given at the $t$-th iteration, e.g.\ $S^t=\sum_{j=0}^kx^t_j$, as well as $x^{t+1}_{k-l},\ldots,x^{t+1}_{k+l}$ for some $l\in\{0,\ldots,k-1\}$. Then one computes $x^{t+1}_{k+l+1}$ recursively (in $l$) as the solution to
\[
x^{t+1}_{k+l+1}\Big(S^t-\sum_{i=k-l}^{k}x^{t+1}_i\Big)=r^t_{k+l+1}\,,
\]
and using this value, one computes $x^{t+1}_{k-l-1}$ from
\[
x^{t+1}_{k-l-1}\Big(S^t+\sum_{i=k+1}^{k+l+1}x^{t+1}_i\Big)=r^t_{k-l-1}\,.
\]
Hence starting from the previous iterate $x^t$ and the `middle' value $x^{t+1}_k$ (the computation of which follows from Proposition~\ref{prop:neven}), the above two equations produce the other $x^{t+1}_j$.
\medskip\\
We proceed with some properties of Algorithm~\ref{algo}. The portmanteau Proposition~\ref{prop:properties} below summarizes some useful properties of the algorithm, in particular it quantifies the update gain at each iteration step.

\begin{proposition}\label{prop:properties}
The iterates $x^t,\, t\ge 0$, of Algorithm~\ref{algo} satisfy the following properties.
\begin{enumerate}[itemsep=-.1em,topsep=-.2em,label=(\roman{*}),labelsep=0.9em]
\item\label{item:hpositive}
If $x^0>0$ componentwise, then $x^t>0$ componentwise, for all $t>0$.
\item\label{item:simplex} The iterates
$x^t$ belong  to the set $\mathcal{S}=\{x\in\rr^{n+1}_+: \sum_{i=0}^n (x*x)_i=\sum_{i=0}^n y_i\}$ for all $t>0$, in agreement with Proposition~\ref{prop:sum}. Additionally, the iterates $x^t$ also satisfy $\sum_{j=0}^nx^t_j\nabla_j\ii(x^t)=0$ for $t\geq 1$.
\item\label{item:Wt+1}
$\ii(y||x^t* x^t)$ decreases at each iteration, in fact one has
\begin{equation}\label{eq:gain1}
\ii(y||x^t* x^t) - \ii(y||x^{t+1}* x^{t+1}) = \ii(\by^t||\by^{t+1}) + \ii(\bw^{t+1}||\bw^t)\ge 0\,,
\end{equation}
where one can use the expression
\begin{align*}
\ii(\bw^{t+1}||\bw^t) & =2\sum_jr^t_j\log \frac{x^{t+1}_j}{x^t_j}\,.
\end{align*}
As a corollary,  $\ii(\bw^{t+1}||\bw^t)$ vanishes asymptotically.
\item
If $y=x^t* x^t$ then $x^{t+1}=x^t$, i.e.\ perfect matches are fixed points of the algorithm.
\item
The update relation  can be written in a formula --albeit an implicit expression-- in terms of the gradient, \begin{equation}\label{eq:xnabla}
x^{t+1}_j\sum_{i=0}^{n-j}x^{t+1}_i=x^t_j\left(-\half\nabla_j\ii(x^t)+\sum_{i=0}^{n-j}x^t_{i}\right)\,,\quad j=0,\ldots, n.
\end{equation}
\item
If $\nabla \ii(x^t)=0$ then $x^{t+1}=x^t$, i.e.\ stationary points of $\ii(x)$ are fixed points of the algorithm. Additionally, if $x^t_j=0$, then also $x^{t+1}_j=0$.
\end{enumerate}
\end{proposition}

\begin{proof}
(i) is obvious.

(ii) Recall that $x^{t+1}$ is for $t\geq 0$ the solution to \eqref{eq:xt}. Summing over $j$ one obtains $\sum_{j=0}^nx^{t+1}_j\sum_{i=0}^{n-j}x^{t+1}_i=\sum_{i=0}^n y_i$ as the row sums of $\by^t$ are the $y_i$. By swapping the summation order, one sees $\sum_{j=0}^nx^{t+1}_j\sum_{i=0}^{n-j}x^{t+1}_i=\sum_{i=0}^n(x^{t+1}*x^{t+1})_i$. The first result follows. Using this in \eqref{eq:xnabla} one obtains $\sum_{j=0}^nx^t_j\nabla_j\ii(x^t)=0$ for $t\geq 1$, as desired.

(iii) Recall that $x^{t+1}_j\sum_{i=j}^n x^{t+1}_{i-j}=r^t_j$ for all $j$, which follows from Equation~\eqref{eq:xstarr} with $r=r^t$ and $x=x^{t+1}$.  Property~(ii), the iterates $x^t$ belong to the set $\mathcal{S}$, allows us to write, also using $r^t_j$ as in \eqref{eq:xt},
\begin{align*}
\ii(\bw^{t+1}||\bw^t) & = \sum_{i=0}^n\sum_{j=0}^i x^{t+1}_jx^{t+1}_{i-j}\log \frac{x^{t+1}_jx^{t+1}_{i-j}}{x^{t}_jx^{t}_{i-j}} \\
& = \sum_{i=0}^n\sum_{j=0}^i x^{t+1}_jx^{t+1}_{i-j}(\log \frac{x^{t+1}_j}{x^{t}_j}+\log\frac{x^{t+1}_{i-j}}{x^{t}_{i-j}}) \\
& = \sum_{j=0}^nx^{t+1}_j\sum_{i=j}^n x^{t+1}_{i-j}\log \frac{x^{t+1}_j}{x^{t}_j}+\sum_{i=0}^n\sum_{l=0}^i x^{t+1}_{i-l} x^{t+1}_{l}\log\frac{x^{t+1}_{l}}{x^{t}_{l}} \\
& = \sum_{j=0}^n r^t_j\log \frac{x^{t+1}_j}{x^{t}_j}+\sum_{l=0}^nx^{t+1}_{l}\sum_{i=l}^n x^{t+1}_{i-l} \log\frac{x^{t+1}_{l}}{x^{t}_{l}} \\
& = \sum_{j=0}^n r^t_j\log \frac{x^{t+1}_j}{x^{t}_j}+\sum_{l=0}^nr^t_l \log\frac{x^{t+1}_{l}}{x^{t}_{l}} \\
& = 2\sum_{j=0}^n r^t_j\log \frac{x^{t+1}_j}{x^{t}_j}\,.
\end{align*}
It follows from \eqref{eq:gain1} that the sum $\sum_{t=0}\ii(y||x^t* x^t)-\ii(y||x^{t+1}* x^{t+1})$ is finite and so is then $\sum_{t=0}^\infty \ii(\bw^{t+1}||\bw^t)$, which implies $\ii(\bw^{t+1}||\bw^t)\to 0$ for $t\to\infty$.

(iv) In this case one has $\rho_j=1$ for all $j$. Hence, \eqref{eq:xt} reads
\[
x^{t+1}_j\sum_{i=0}^{n-j}x^{t+1}_i=\sum_{i=j}^nx_j^tx^t_{i-j}\,,\quad j=0,\ldots, n.
\]
As this system of equations has a unique solution in terms of the variables $x^{t+1}_j$ by virtue of Propositions~\ref{prop:neven} and \ref{prop:nodd}, the solution is necessarily $x^{t+1}_j=x^{t}_j$ for all $j$.

(v) In view of \eqref{eq:nablarho} the sum $\sum_{i=j}^n\by^t_{ij}=x^t_j\sum_{i=j}^nx^t_{i-j}\rho_i^t$ in \eqref{eq:xt} can be written as $x^t_j\big(-\half\nabla_j\ii(x^t)+\sum_{i=j}^nx^t_{i-j}\big)$, hence the  $x^{t+1}_j$ also satisfy \eqref{eq:xnabla}.

(vi) If some $x^t_m=0$, then by the multiplicative nature of the algorithm (see \eqref{eq:xnabla}), also $x^{t+1}_m=0$, so in the further analysis we assume that all $x^t_m>0$.
By Propositions~\ref{prop:neven} and \ref{prop:nodd} the $x^{t+1}_j$ are unique solutions to the system of equations. If $\nabla \ii(x^t)=0$, then the equations to solve become
\[
x^{t+1}_j\sum_{i=0}^{n-j}x^{t+1}_i=x^t_j\sum_{i=0}^{n-j}x^t_{i}\,,\quad j=0,\ldots n,
\]
and therefore $x^{t+1}=x^t$.
\end{proof}
Knowing properties of Algorithm~\ref{algo} we now present a proof of Proposition~\ref{proposition:exist}.
\begin{proof}\!\!\!\textbf{ of Proposition~\ref{proposition:exist}.}
As $y_0>0$, one can only have finite divergence if $x_0>0$, because of the term $y_0\log\frac{y_0}{(x*x)_0}$ in the divergence. Furthermore, $\sum_{k=0}^n(x*x)_k\geq x_0\sum_{k=0}^nx_k$. Hence if some $x_k\to\infty$, the divergence tends to infinity as well. Therefore, the minimum has to be sought in a sufficiently large compact $K_0$ set of $\rr^{n+1}_+$. We use Proposition~\ref{prop:sum} and various results from Proposition~\ref{prop:properties}. Let $x=x^0$ be an arbitrary point in $\rr^{n+1}_+$. Then one step of Algorithm~\ref{algo} yields $x^1\in\rr^{n+1}_+$ satisfying  $\ii(x^1)\leq \ii(x)$. Hence we can limit $K_0$ to those $x$ which also satisfy $\sum_{j=0}^n(x*x)_j=\sum_{j=0}^ny_j$, which gives a compact set as well, call it $K_1$. Note that $\ii(x)=\sum_{k:y_k>0}(y_k\log\frac{y_k}{(x*x)_k}-y_k)+\sum_k(x*x)_k$. To find the minimum, we can restrict ourselves even further to those $x$ for which $(x*x)_k\geq \eps$ for all $k$ such that $y_k>0$, by choosing $\eps$ sufficiently small and positive. This implies that we restrict the finding of the minimizers to an even smaller compact set $K_2$ on which $\ii$ is continuous. This proves the existence of a minimizer.
\end{proof}
The Algorithm~\ref{algo} shares some properties with its counterpart in \cite{fs2021inverse}.  We mention only a few  related
to its large running time behavior. We follow ideas taken from  \cite{finessospreij2015ieeeit} and \cite{fs2021inverse}.
\begin{lemma}\label{lemma:limfix}
Let all data $y_i$ be positive.
Limit points of the sequence $(x^t)$ are fixed points of Algorithm~\ref{algo}. Moreover, if along a subsequence $(t_j)$ one has the convergence $x^{t_j}\to x^\infty$, then also $x^{t_j+1}\to x^\infty$.
\end{lemma}

\begin{proof}
Let $\ii^t=\ii(y||x^t*x^t)$. Recall from Proposition~\ref{prop:properties}\ref{item:Wt+1} that the divergence  $\ii(\bw^t||\bw^{t+1})\to 0$. Let $x^\infty$ be a limit point of the algorithm and $x^{\infty+1}$ its next iterate, with corresponding $\bw^\infty$ and $\bw^{\infty+1}$, then by continuity of the divergence in both arguments $\ii(\bw^\infty||\bw^{\infty+1})=0$, so $\bw^\infty=\bw^{\infty+1}$. We have for the $00$-elements $\bw^\infty_{00}=\bw^{\infty+1}_{00}$, so $x^\infty_0=x^{\infty+1}_0\geq 0$.
As we have assumed $y_0>0$, we must have $x^\infty_0=x^{\infty+1}_0>0$, otherwise $\ii(y_0||(x^\infty*x^\infty)_0)=\ii(y_0||(x^\infty_0)^2)$ would be infinite. Because $\bw^\infty_{0j}=x^\infty_0x^\infty_j$ and $\bw^{\infty+1}_{0j}=x^{\infty+1}_0x^{\infty+1}_j$ are equal we can now conclude $x^\infty_j=x^{\infty+1}_j$ for $j>0$.

The second assertion follows from the just given proof.
\end{proof}
Numerical examples, see Section~\ref{section:numerics}, indicate that depending on the initial values of the algorithm convergence of the iterates takes place to a point that is a local minimizer of the objective function. Comparing the full support Problem~\ref{problem} to the half support problem discussed in \cite{fs2021inverse}, one observes that in the present case there are about twice as many parameters involved when the number of observations is kept the same. Therefore it is not surprising that multiple local minima exist and that convergence of the algorithm to a unique limit, that should also be a global minimizer is at least questionable.
The next theorem is therefore the best possible result on the convergence properties of the algorithm.

\begin{theorem}\label{thm:kt}
Limit points of the sequence $(x^t)$ are Kuhn-Tucker points of the minimization Problem~\ref{problem}.
\end{theorem}

\begin{proof}
The main idea here is to reason as in the proof of Proposition~4.8 in \cite{fs2021inverse}. First we invoke the second assertion of Lemma~\ref{lemma:limfix} to deduce by taking limits in \eqref{eq:xnabla} that
\[
x^{\infty}_j\sum_{i=0}^{n-j}x^{\infty}_i=x^\infty_j\Big(-\half\nabla_j\ii(x^\infty)+\sum_{i=0}^{n-j}x^\infty_{i}\Big)\,,\quad j=0,\ldots, n.
\]
Hence one has
\[
x^\infty_j\nabla_j\ii(x^\infty)=0\,, \quad j=0,\ldots, n.
\]
If $x^\infty_j>0$, then necessarily $\nabla_j\ii(x^\infty)=0$. If $x^\infty_j=0$, the analysis is more delicate to deduce that $\nabla_j\ii(x^\infty)\geq 0$ and analogous to the proof of Proposition~4.8 in \cite{fs2021inverse}. Further details are omitted.
\end{proof}

\section{Numerical examples}\label{section:numerics}

In this section we illustrate the behaviour of Algorithm~\ref{algo} reviewing the results of numerical experiments.  All figures are collected at the end of the paper.
In the first experiment the data sequence $y$, with $n=20$, was obtained as the autoconvolution $y=x*x$ of a randomly generated signal $x$, with components $x_j$ uniformly distributed in the (real) interval $[1, 10]$. We investigated whether in this "exact" case the algorithm is capable of retrieving the true sequence $x$ generating the data $y=x*x$. Figs.~\ref{fig:exact_1}, \ref{fig:exact_2}, and \ref{fig:exact_3} share the same data sequence $y$, obtained as specified above. The three figures show the results of three runs of Algorithm~\ref{algo}, with $T=2000$, started from three different (randomly generated) initial conditions $x^0$. In each figure the top graph shows, in distinct colors, the trajectories of the iterates of the components $x^t_j$, plotted against the iteration number  $t\in[1, T]$. The diamonds at the right end of the graph are the true $x_j$ values that generated the data $y$. The second graph shows the superimposed plots of the data generating signal $x$, and of the reconstructed signal $x^T$, at the last iteration, both plotted against their component number (in the figure labelled $j=1,2, \dots, n+1$ instead of $j=0,1, \dots, n$). The third graph shows the decreasing sequence $\ii(y||x^t* x^t)$.  The fourth and last graph shows the superimposed plots of the data vector $y$ and of the reconstructed convolution $x^T* x^T$, at the last iteration, both against the component number (labelled $j=1,\dots, n$). Fig.~\ref{fig:exact_1} shows a typical behavior, with the Algorithm producing limit points $x^\infty$ that coincide with the generating vector $x$. Another common behavior is shown in Fig.~\ref{fig:exact_2} which shows convergence (the iterates in the top graph stabilize for $t\rightarrow \infty$) but the limit point $x^\infty$ does not coincide with the generating $x$. An atypical but interesting behavior is shown in Fig.~\ref{fig:exact_3} where, after approximately 1600 iterations, the iterates $x^t$ suddenly switch basin of attraction as indicated by the sudden dip in the divergence $\ii(y||x^t* x^t)$.
\medskip\\
For the second experiment the data $y$ are not produced by autoconvolution but randomly generated. The elements of $y$ are defined as $y_k=(k+1)u_k$, for $k=0,\ldots, n$, where the $u_k$ are generated as IID random numbers from a uniform distribution on $[1,2 K^2]$, whose expectation is approximately $K^2$. Here $K$ is an integer, chosen with a convenient value. This procedure is motivated by comparison with the first experiment, the exact case in which $y=x*x$ with the $x_k$'s randomly generated as uniform on $[1,2K]$. In that case the expectation of $x_ix_j$ (for $i\neq j$) is roughly equal to $K^2$ and the autoconvolution $(x*x)_k$ will then have expectation approximately equal to $(k+1)K^2$. The choice to generate the $y_k$ as described above has been made to generate components $y_k$ which have an upward trend, as is the case when $y_k=(x*x)_k$. It would not make a lot of sense to model observations $y_k$ as the output of an autoconvolution system, if e.g.\ the components $y_k$ show a decreasing behavior. Figs.~\ref{fig:random_1} and \ref{fig:random_2} share the same data sequence $y$, generated as specified above with $n=12$ and $K=5$. The two figures show the results of two runs of Algorithm~\ref{algo}, with $T=2000$, started from two different (randomly generated) initial conditions. In each figure the top graph shows, in distinct colors, the trajectories of the iterates of the components $x^t_j$, plotted against the iteration number  $t\in[1, T]$. The second graph shows the decreasing sequence $\ii(y||x^t* x^t)$.  The third and last graph shows the superimposed plots of the data vector $y$ and of the fitted convolution $x^T* x^T$, at the last iteration, both against the component number (labelled $j=1,\dots, n$). Fig.~\ref{fig:random_1} shows a typical behavior of the algorithm. The iterates $x_k^t$ converge quickly to a limit point $x^\infty$, as shown by their stabilizing trajectories in the top graph and quick decrease in the divergence $\ii(y||x^t* x^t)$.
A different run on the same random data $y$, starting from different initial conditions $x^0$, produces Fig.~\ref{fig:random_2}. Also in this case the iterates $x_k^t$ (top graph) stabilize for $t\rightarrow \infty$, but at a much slower rate, as indicated by the divergence which only after about 1000 iterations dips to reach a limit value much higher than in the first run. The numerical behaviors, both in the exact and in the random data cases confirm the strong dependence of the limit points of the algorithm from the initial conditions, as is expected from this class of algorithms.

\section{Conclusion}

We have studied the problem of optimal approximation of a given nonnegative signal  with the scalar autoconvolution  of a nonnegative signal, where the signals are of equal time length. Because of the nonnegativity constraints, the $\ii$-divergence has been adopted as optimality criterion. We have derived an iterative descent algorithm of the alternating minimization type, of which we studied the asymptotic behavior, and proved that its limit points are Kuhn-Tucker points of the original minimization problem.

\section*{Acknowledgement}

We are very grateful to the reviewer whose constructive comments and suggestions were very helpful to improve our contribution.


\newpage                                       
\begin{appendices}                             
\renewcommand{\thesection}{\Alph{section}}    

\section{Solving minimization Problem~\ref{problemw}}\label{app:equation}

The solution to the minimization Problem~\ref{problemw} was given in Lemma~\ref{lemmaw}. The lemma states that in order to construct the solution one has to solve the system of equations \eqref{eq:xstarr}, for the special choice of the $r_j$ as in \eqref{eq:ry}. For convenience we repeat the system of equation below as \eqref{eq:xj},
%
\begin{equation}\label{eq:xj}
x_j\sum_{i=0}^{n-j}x_i=r_j\,, \quad j=0,\ldots,n,
\end{equation}
 In this appendix we detail the solution, denoted $x_j^\star$ for $j=0,\dots n$, of system \eqref{eq:xj}, for any given set of positive $r_j$.
As the analytic expressions for the solution turn out to be different for $n$ even and odd, we treat the two cases separately. First some preliminaries that are generally valid.

\begin{lemma}\label{lemma:R}
Consider the system of equations \eqref{eq:xj}. Assume
\begin{equation}\label{eq:s2}
S^2:=\sum_{j\leq n/2}r_j-\sum_{j>n/2}r_j\geq 0\,,
\end{equation}
and $S=\sqrt{S^2}$. Let  $x_j^\star$ be the solutions. Then it holds that $H_n:=\sum_{j\leq n/2}x^\star_j=S$.
\end{lemma}

\begin{proof}
For convenience we write $k=\lfloor \frac{n}{2} \rfloor$, so $S^2=\sum_{j\leq k}r_j-\sum_{j\geq k+1}r_j$. Add now the first $k+1$ equations of \eqref{eq:xj} and subtract the last $n-k$. The resulting right hand side is then exactly $S^2$ and the left hand side becomes
\[
H_n^2:=\sum_{j=0}^k\sum_{i=0}^{n-j} x_jx_i - \sum_{j=k+1}^n\sum_{i=0}^{n-j} x_jx_i\,.
\]
Consider the first double sum in $H_n^2$, which is equal to
\[
\sum_{i=0}^k\sum_{j=0}^{k} x_jx_i+\sum_{i=k+1}^n\sum_{j=0}^{n-i} x_jx_i=\Big(\sum_{i=0}^kx_i\Big)^2+\sum_{i=k+1}^n\sum_{j=0}^{n-i} x_jx_i\,,
\]
which yields for the solutions $x_j^\star$ the identity $H_n^2=\big(\sum_{i=0}^kx^\star_i\big)^2=S^2$.
\end{proof}
Below we need for $j\in\{0,\ldots,n\}$ the quantities
\begin{align}
B_j&=\sum_{i=0}^jr_i\,,\label{eq:bj} \\
E_j&=\sum_{i=j}^nr_i\,.\label{eq:ej}
\end{align}

\subsection{Case of $n$ even} \label{section:neven}

Assume $n$ is even, $n=2k$. Consider the equation in \eqref{eq:xj} for $j=k$,
\[
x_k\sum_{i=0}^{n-k}x_i=r_k\,.
\]
Lemma~\ref{lemma:R} gives $x_k^\star=r_k/S$, provided $S>0$.  The complete solution to \eqref{eq:xj} is the content of Proposition~\ref{prop:neven}. Before formulating the results of that proposition in full detail, we outline the underlying structure. Starting from the `middle' term $x^\star_k=\frac{r_k}{S}$ one sets up two recursions, for $x^\star_{k+l}$ and for $x^\star_{k-l}$, with $0\leq l \leq k$ that are solved in an alternating way.
The explicit solution is presented in the proposition.
\begin{proposition}\label{prop:neven}
Assume $n=2k$ and strenghten condition \eqref{eq:s2} to $S^2>0$. Then the system of equations \eqref{eq:xj} has the unique solutions $x^\star_j$ given by the following expressions. For $j=k$ one has $x^\star_{k}=\frac{r_{k}}{S}$.
For $m>k$ one has
\begin{equation}\label{eq:m>k}
x^\star_{m} = \frac{r_{m}S}{B_{k}-E_{k+1}}\prod_{i=k+1}^m\frac{B_{n-i+1}-E_{i}}{B_{n-i}-E_{i}}\,,
\end{equation}
and for $m\leq k$
\begin{equation}\label{eq:m<k}
x^\star_{m} = \frac{r_{m}}{S}\prod_{i=1}^{k -m}\frac{B_{k-i}-E_{k+i}}{B_{k-i}-E_{k+1+i}}\,.
\end{equation}
\end{proposition}

\begin{proof}
Recall $n=2k$. It follows from Lemma~\ref{lemma:R} that for the solutions one necessarily has $\sum_{j\leq k}x^\star_j=S$. The following procedure shows uniqueness.  Knowing $x_k^\star$, we look at the equation for $x_{k+1}$, rewritten as $x_{k+1}(S-x^\star_k)=r_{k+1}$, which gives $x^\star_{k+1}$. Then we consider $x_{k-1}(S+x^\star_{k+1})=r_{k-1}$, from which we can now find $x^\star_{k-1}$. Iterating this procedure gives all solutions $x_j^\star$.

Having shown the uniqueness, we now verify that the $x_m^\star$ given by Equations~\eqref{eq:m>k} and~\eqref{eq:m<k} satisfy the equations.
First we claim for $p\leq k$
\begin{equation}\label{eq:rmin}
S-\sum_{l=0}^px^\star_{k-l}=\frac{B_{k-1}-E_{k+1}}{S}\prod_{l=1}^p\frac{B_{k-l-1}-E_{k+l+1}}{B_{k-l}-E_{k+l+1}}\,.
\end{equation}
One verifies this equality easily for $p=0$ by application of Lemma~\ref{lemma:R}, and we proceed by induction. Suppose this formula holds true for some $p-1$, where $p\geq 1$. Then we compute $S-\sum_{l=0}^px^\star_{k-l}= S-\sum_{l=0}^{p-1}x^\star_{k-l}-x^\star_{k-p}$, which becomes
\[
\frac{B_{k-1}-E_{k+1}}{S}\prod_{l=1}^{p-1}\frac{B_{k-l-1}-E_{k+l+1}}{B_{k-l}-E_{k+l+1}} \,-\, \frac{r_{k-p}}{S}\prod_{i=1}^p\frac{B_{k-i}-E_{k+i}}{B_{k-i}-E_{k+1+i}}\,,
\]
and this equals
\[
\frac{B_{k-1}-E_{k+1}}{S}\left(\prod_{l=1}^{p-1}\frac{B_{k-l-1}-E_{k+l+1}}{B_{k-l}-E_{k+l+1}} - \frac{r_{k-p}}{B_{k-p}-E_{k+1+p}}\prod_{i=2}^{p}\frac{B_{k-i}-E_{k+i}}{B_{k-i-1}-E_{k+i}}\right)\,,
\]
which in turn is equal to
\[
\frac{B_{k-1}-E_{k+1}}{S}\prod_{l=1}^{p-1}\frac{B_{k-l-1}-E_{k+l+1}}{B_{k-l}-E_{k+l+1}}\left(1 - \frac{r_{k-p}}{B_{k-p}-E_{k+1+p}}\right)\,.
\]
The last expression equals
\[
\frac{B_{k-1}-E_{k+1}}{S}\prod_{l=1}^{p-1}\frac{B_{k-l-1}-E_{k+l+1}}{B_{k-l}-E_{k+l+1}}\times\frac{B_{k-p-1}-E_{k+1+p}}{B_{k-p}-E_{k+1+p}}\,,
\]
since $r_{k-p}=B_{k-p}-B_{k-p-1}$. But this is just \eqref{eq:rmin} and proves the claim.

Another computation gives
\begin{equation}
x^\star_{k+l}  = \frac{r_{k+l}S}{B_{k-1}-E_{k+1}}\prod_{i=1}^{l-1}\frac{B_{k-i}-E_{k+i+1}}{B_{k-i-1}-E_{k+i+1}}\,.\label{eq:xplus}
\end{equation}
We can now compute the product $(S-\sum_{l=0}^{p-1}x^\star_{k-l})x^\star_{k+p}$. From \eqref{eq:rmin} and \eqref{eq:xplus} we see that the product becomes $r_{k+p}$, as desired.

In the next part of the proof we need $S+\sum_{l=1}^px^\star_{k+l}$, $p\geq 0$, which we claim to be equal to
\begin{equation}\label{eq:rplus}
S\prod_{l=1}^p\frac{B_{k-l}-E_{k+l+1}}{B_{k-l}-E_{k+l}}\,.
\end{equation}
As above, this can be proved by induction. It is trivial to check this for $p=1$. Assuming the identity to hold for $p-1$ ($p\geq 1$), we compute
\begin{align*}
S+\sum_{l=1}^px^\star_{k+l} & = S\prod_{l=1}^{p-1}\frac{B_{k-l}-E_{k+l+1}}{B_{k-l}-E_{k+l}}+x^\star_{k+p} \\
& = S\prod_{l=1}^{p-1}\frac{B_{k-l}-E_{k+l+1}}{B_{k-l}-E_{k+l}}+\frac{r_{k+p}}{B_{k-1}-E_{k+1}}\prod_{i=1}^{p-1}\frac{B_{k-i}-E_{k+i+1}}{B_{k-i-1}-E_{k+i+1}} \\
& = S\prod_{l=1}^{p-1}\frac{B_{k-l}-E_{k+l+1}}{B_{k-l}-E_{k+l}}\left(1+\frac{r_{k+p}}{B_{k-1}-E_{k+1}}\frac{B_{k-1}-E_{k+1}}{B_{k-p}-E_{k+p}}\right) \\
& = S\prod_{l=1}^{p-1}\frac{B_{k-l}-E_{k+l+1}}{B_{k-l}-E_{k+l}}\left(1+\frac{r_{k+p}}{B_{k-p}-E_{k+p}}\right) \\
& = S\prod_{l=1}^{p-1}\frac{B_{k-l}-E_{k+l+1}}{B_{k-l}-E_{k+l}}\frac{B_{k-p}-E_{k+p+1}}{B_{k-p}-E_{k+p}}\,,
\end{align*}
which proves the claim. The next fact to verify is the value of the product $x^\star_{k-p}\big(S+\sum_{l=1}^px^\star_{k+l}\big)$. For that we need a result in the opposite direction (compared to \eqref{eq:xplus}). For $l\geq 1$ one has
\begin{equation}\label{eq:xmin}
x^\star_{k-l} = \frac{r_{k-l}}{S}\prod_{i=1}^l\frac{B_{k-i}-E_{k+i}}{B_{k-i}-E_{k+1+i}}\,.
\end{equation}
A combination of \eqref{eq:rplus} and \eqref{eq:xmin} shows that this product equals $r_{k-p}$.
\end{proof}

\subsection{Case of $n$ odd}

Here we consider the case where $n$ is odd,
$n=2k+1$. Recall the result of Lemma~\ref{lemma:R}, $\sum_{j=0}^k x_j^\star=S$, with now $S^2=\sum_{j=0}^kr_j-\sum_{j=k+1}^nr_j=B_k-E_{n-k}$. The equation in \eqref{eq:xstarr} for $j=k+1$ is $x_{k+1}\sum_{i=0}^kx_i=r_{k+1}$. Hence $x_{k+1}^\star=\frac{r_{k+1}}{S}$, under the assumption that $S>0$. The procedure to get the other solutions $x_j^\star$ is as in Section~\ref{section:neven}, again a double system of recursions is set up.
For instance one gets $x_k^\star=\frac{r_k}{S+x_{k+1}^\star}=r_k\frac{S}{S^2+r_{k+1}}=r_k\frac{S}{B_{k}-E_{n-k+1}}$, and $x_{k+2}^\star=\frac{r_{k+2}}{S-x_{k}^\star}=\frac{r_{k+2}}{S}\frac{B_k-E_{n-k+1}}{B_{k-1}-E_{n-k+1}}$. The general case is detailed in the following proposition

\begin{proposition}\label{prop:nodd}
Let $n=2k+1$ and $S^2>0$. Then $x_{k+1}^\star=\frac{r_{k+1}}{S}$. For $m\geq k+1$ one has
\begin{equation}\label{eq:m>kodd}
x_m^\star=\frac{r_m}{S}\prod_{l=1}^{m-k-1}\frac{B_{k-l+1}-E_{k+l+1}}{B_{k-l}-E_{k+l+1}}\,,
\end{equation}
and for $m\leq k$ one has (with $E_{n+1}=0$)
\begin{equation}\label{eq:m<kodd}
x_m^\star=\frac{r_mS}{B_k-E_{k+2}}\prod_{l=1}^{k-m}\frac{B_{k-l}-E_{k+l+1}}{B_{k-l}-E_{k+l+2}}\,.
\end{equation}
\end{proposition}

\begin{proof}
Let the $x_m^\star$ be given by Equations~\eqref{eq:m>kodd} and~\eqref{eq:m<kodd}.
Instead of Equation~\eqref{eq:rmin} we now have for $p\leq k$
\[
S-\sum_{l=0}^px^\star_{k-l}=S\prod_{l=1}^{p+1}\frac{B_{k-l}-E_{k+l+1}}{B_{k-l+1}-E_{k+l+1}}\,.
\]
For $p=0$ this is verified by application of Lemma~\ref{lemma:R}. For other values of $p$ one proves this by induction.

Instead of Equation~\eqref{eq:rplus} we now have for $p\leq k+1$
\[
S+\sum_{l=1}^px^\star_{k+l}=\frac{B_k-E_{n-k+1}}{S}\prod_{l=1}^{p-1}\frac{B_{k-l}-E_{k+l+2}}{B_{k-l}-E_{k+l+1}}\,,
\]
which can again be proved by induction with respect to $p$.

One then verifies that the postulated expressions for the $x_m^\star$ satisfy the equations by arguments similar to those in the proof of Proposition~\ref{prop:neven}.
\end{proof}

\end{appendices}   

\bibliographystyle{plain}

\begin{figure}
\begin{center}
	\includegraphics[viewport=200 0 400 460, scale=0.80]{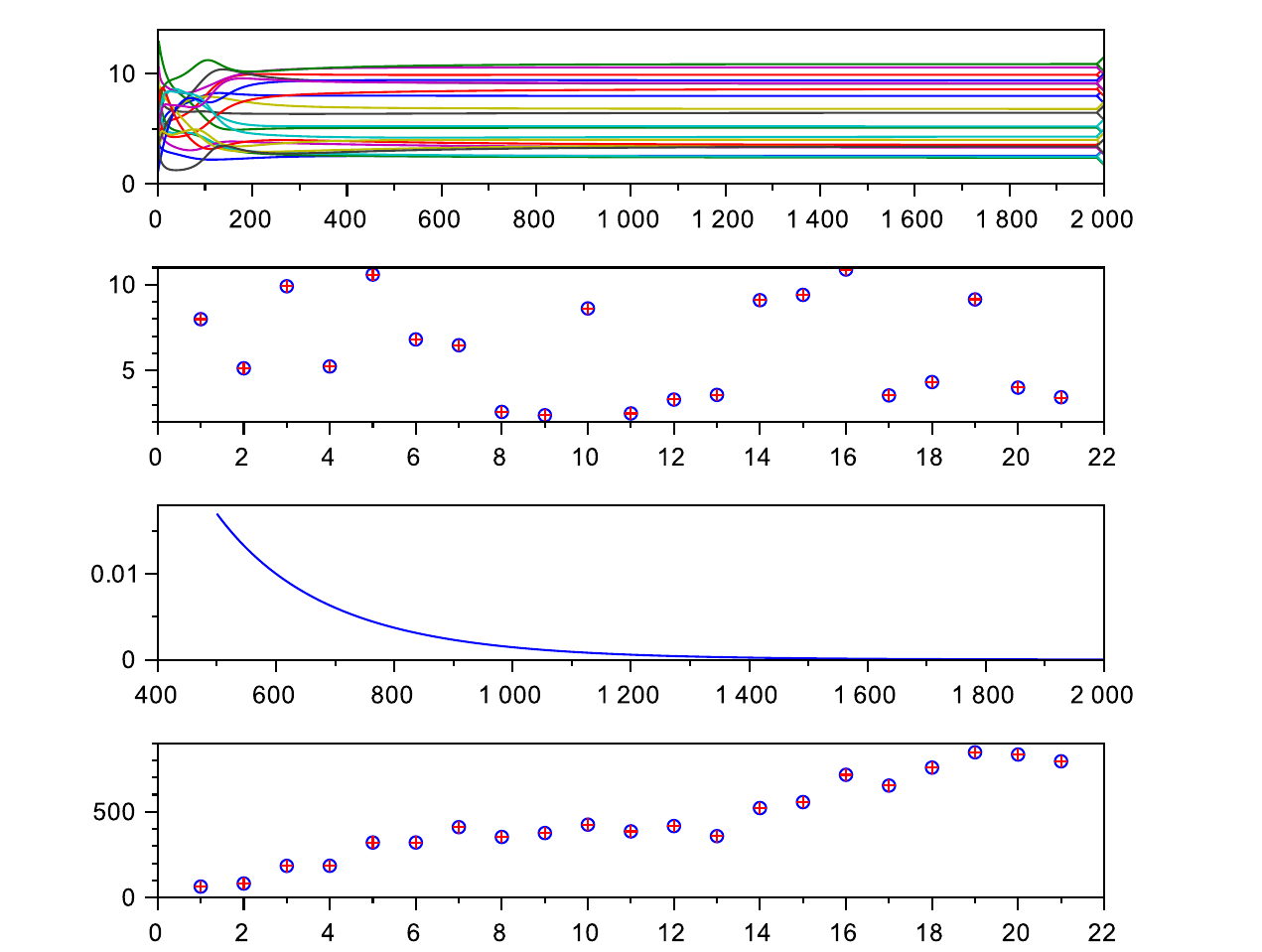}
	\vspace{5ex}
		\caption{True model, $n=20$ and $T=2000$. Top panel: $n+1$ components $x^t_i$ against iteration $t$; the diamonds at $T=2000$ are the true values $x_i$ to which the $x^t_i$ converge. Second panel: $x^T_i$ (plusses) and true values $x_i$ (circles) against $i$. Third panel: $\ii(y||x^t* x^t)$ against $t$. Fourth panel: $y_i$ (circles) and $(x^T* x^T)_i$ (plusses) against $i$.} \label{fig:exact_1}
\end{center}
\end{figure}

\vfill

\begin{figure}
\begin{center}
	\includegraphics[viewport=200 0 400 460, scale=0.80]{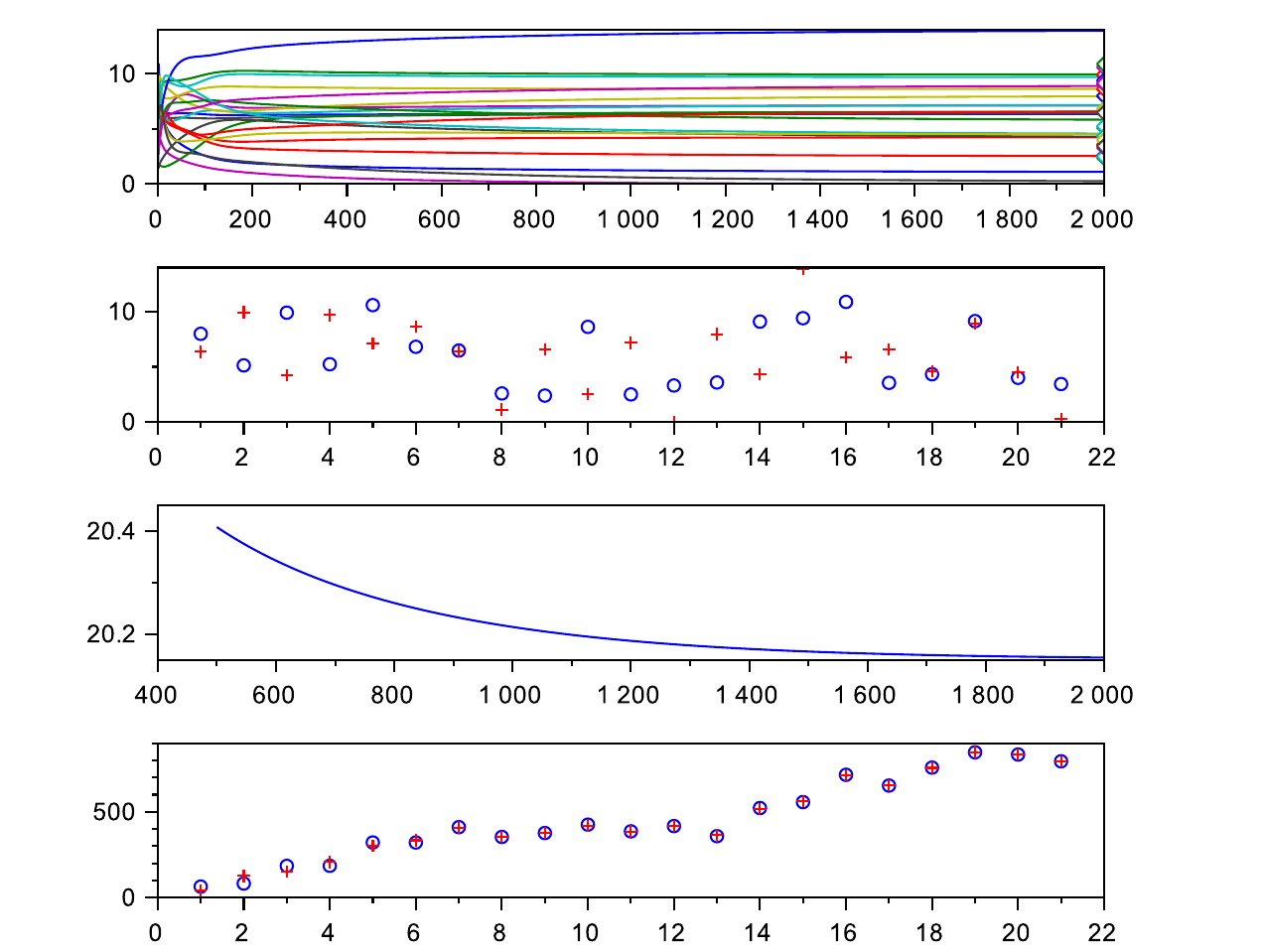}
	\vspace{5ex}
    \caption{The same data $y$ as in Fig.~\ref{fig:exact_1}, with different initial conditions $x^0$. The iterates $x^t$ do not converge to the true $x$. }
		\label{fig:exact_2}
\end{center}
\end{figure}

\begin{figure}
\begin{center}
	\includegraphics[viewport=200 0 400 460, scale=0.80]{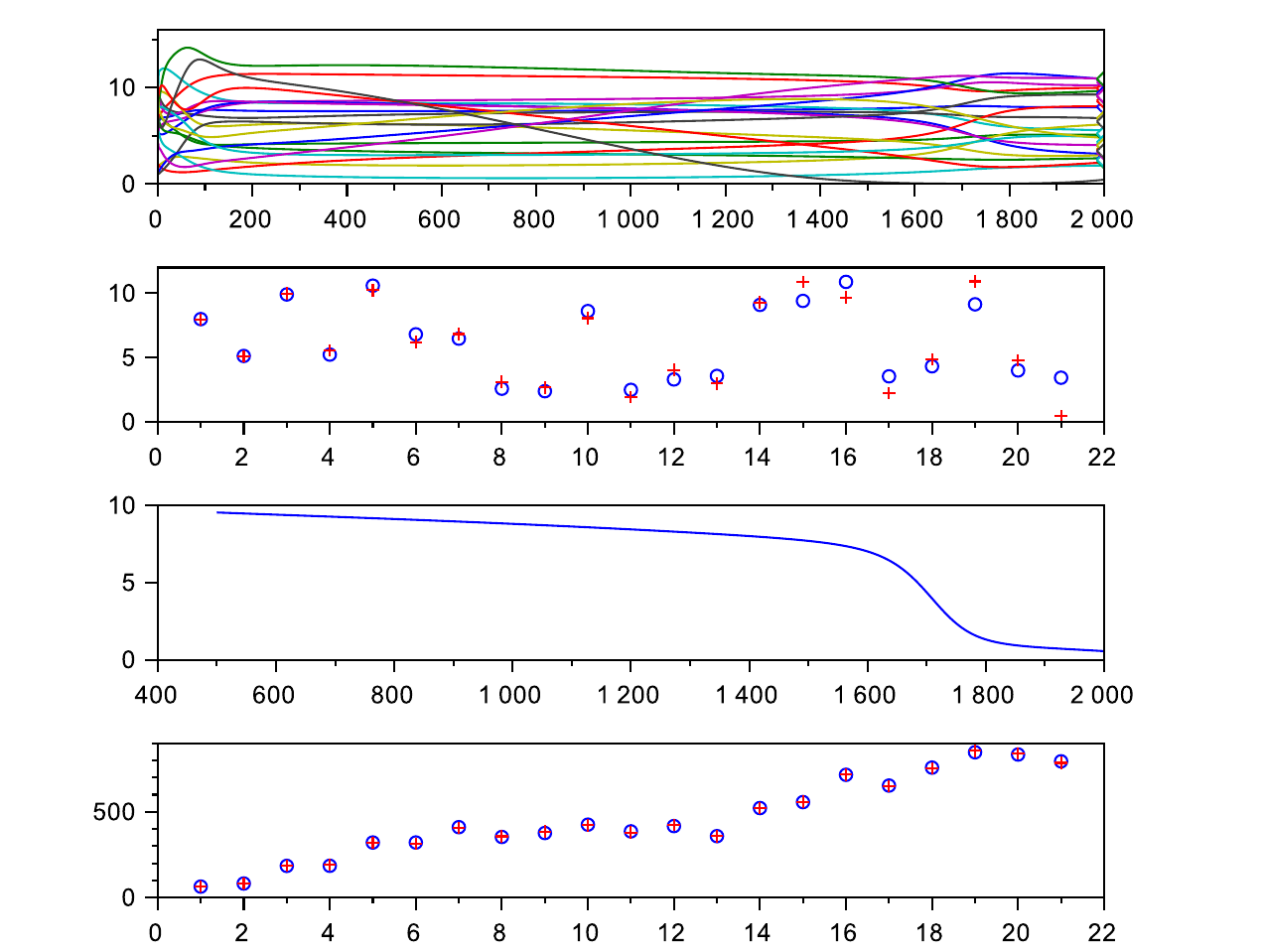}
	\vspace{5ex}
		\caption{The same data $y$ as in Fig.~\ref{fig:exact_1}, with different initial conditions $x^0$. The iterates $x^t$ switch basin of attraction.}
		\label{fig:exact_3}
\end{center}
\end{figure}

\begin{figure}
\begin{center}
	\includegraphics[viewport=200 0 400 460, scale=0.80]{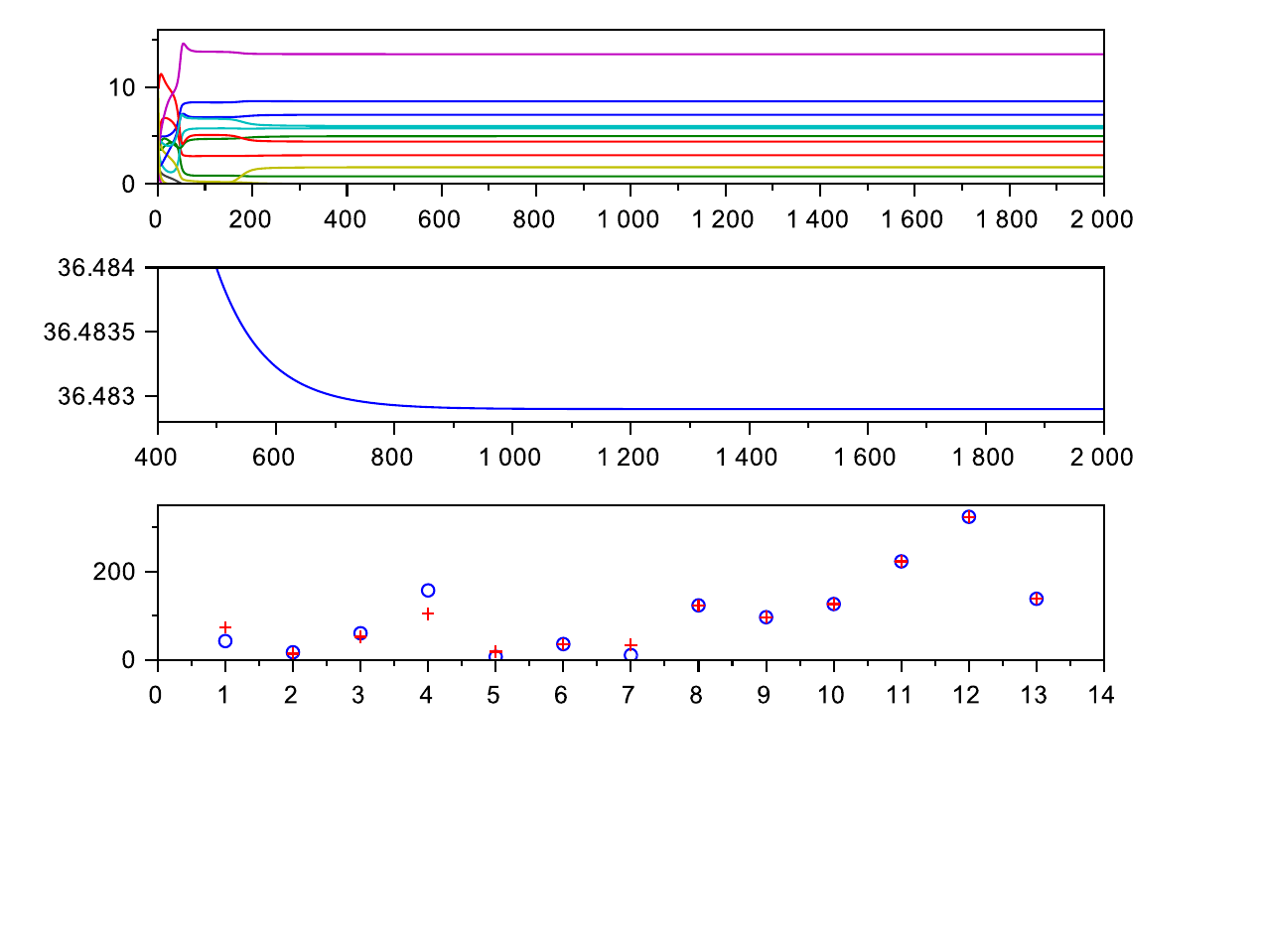}
	\vspace{-15ex}
		\caption{Randomly generated $y$, with $n=12$ and $T=2000$. Top panel: components $x^t_i$ against iteration index $t$. Second panel:  $\ii(y||x^t* x^t)$ against $t$. Third panel: $y_i$ (circles) and final autoconvolutions $(x^T* x^T)_i$ (plusses) against $i$. Third panel: The values of the $y_i$ (circles) and the final autoconvolutions $(x^T* x^T)_i$ (plusses).}
		\label{fig:random_1}
\end{center}
\end{figure}

\begin{figure}
\begin{center}
	\includegraphics[viewport=200 0 400 460, scale=0.80]{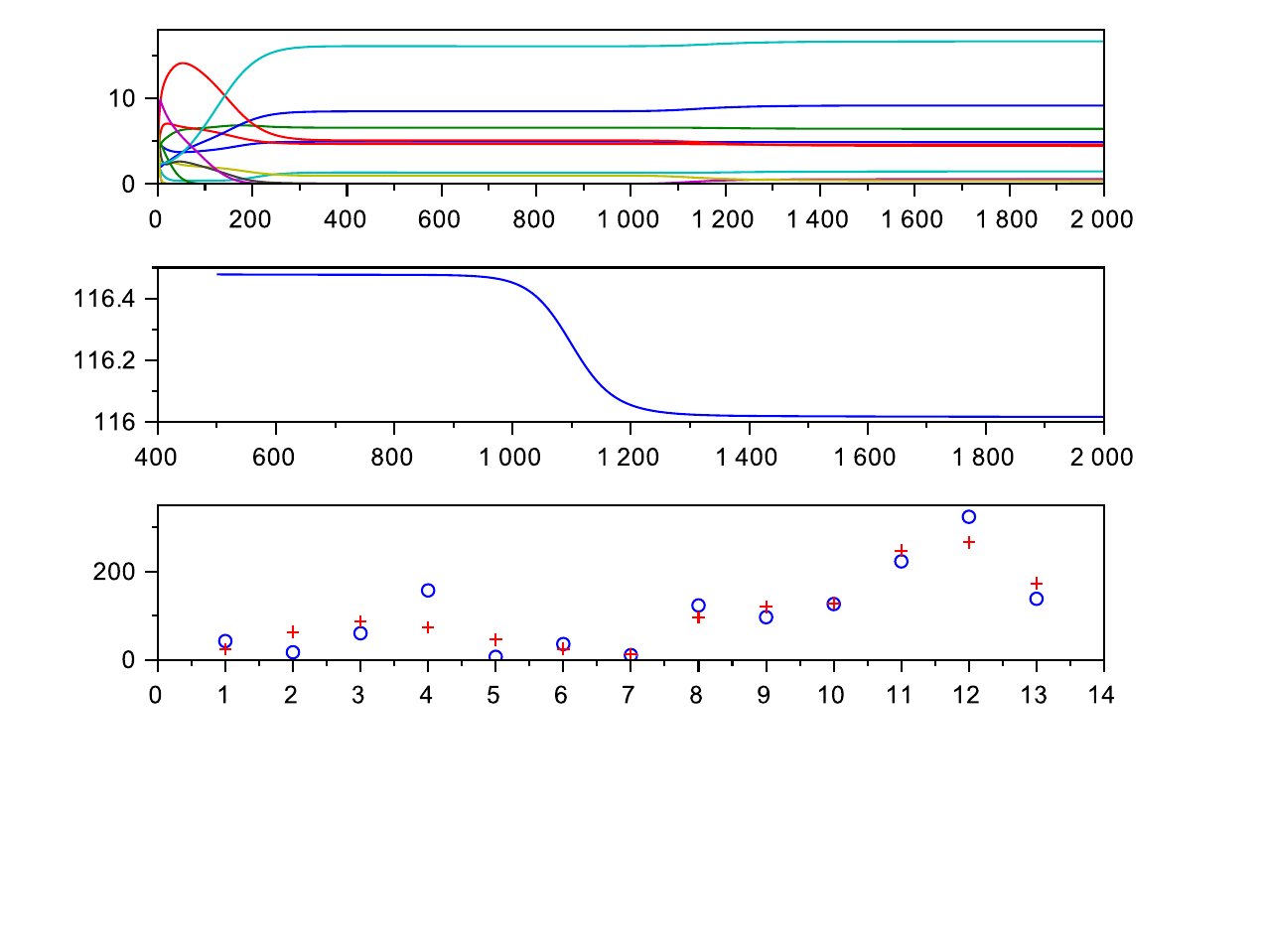}
	\vspace{-15ex}
		\caption{The same data $y$ as in Fig.~\ref{fig:random_1}, with different initial conditions $x^0$.}
		\label{fig:random_2}
\end{center}
\end{figure}

\end{document}